\documentclass{amsart}%
\usepackage{amssymb,setspace}
\usepackage[hmarginratio=1:1]{geometry}%
\usepackage{ifpdf}
\ifpdf
  \usepackage[hyperindex,pagebackref]{hyperref}%
\else
  \expandafter\ifx\csname dvipdfm\endcsname\relax
    \usepackage[hypertex,hyperindex,pagebackref]{hyperref}
  \else
    \usepackage[dvipdfm,hyperindex,pagebackref]{hyperref}
  \fi
\fi
\theoremstyle{plain}
\newtheorem{theorem}{Theorem}[section]
\newtheorem{corollary}{Corollary}[section]
\newtheorem{lemma}{Lemma}[section]
\theoremstyle{definition}
\newtheorem{definition}{Definition}[section]
\theoremstyle{remark}
\newtheorem{remark}{Remark}[section]
\allowdisplaybreaks[4]
\numberwithin{equation}{section}
\DeclareMathOperator{\td}{d\mspace{-1mu}}

\begin{document}

\title[Integral inequalities of Hermite-Hadamard type]
{Some integral inequalities of Hermite-Hadamard type for functions whose derivatives of $\boldsymbol{n}$-th order are $\boldsymbol{(\alpha,m)}$-convex}

\author[F. Qi]{Feng Qi}
\address[Qi]{Department of Mathematics, School of Science, Tianjin Polytechnic University, Tianjin City, 300387, China; College of Mathematics, Inner Mongolia University for Nationalities, Tongliao City, Inner Mongolia Autonomous Region, 028043, China}
\email{\href{mailto: F. Qi <qifeng618@gmail.com>}{qifeng618@gmail.com}, \href{mailto: F. Qi <qifeng618@hotmail.com>}{qifeng618@hotmail.com}, \href{mailto: F. Qi <qifeng618@qq.com>}{qifeng618@qq.com}}
\urladdr{\url{http://qifeng618.wordpress.com}}

\author[M. A. Latif]{Muhammad Amer Latif}
\address[Latif]{College of Science, Department of Mathematics, University of Hail, Hail 2440, Saudi Arabia}
\email{\href{mailto: M. A. Latif <m_amer_latif@hotmail.com>}{m\_amer\_latif@hotmail.com}, \href{mailto: M. A. Latif <m.alatif@uoh.edu.sa>}{m.alatif@uoh.edu.sa}}

\author[W.-H. Li]{Wen-Hui Li}
\address[Li]{Department of Mathematics, School of Science, Tianjin Polytechnic University, Tianjin City, 300387, China}
\email{\href{mailto: W.-H. Li <wen.hui.li@foxmail.com>}{wen.hui.li@foxmail.com}, \href{mailto: W.-H. Li <wen.hui.li102@gmail.com>}{wen.hui.li102@gmail.com}}

\author[S. Hussain]{Sabir Hussain}
\address[Hussain]{Department of Mathematics, University of Engineering and Technology, Lahore, Pakistan}
\email{\href{mailto: S. Hussain <sabirhus@gmail.com>}{sabirhus@gmail.com}}

\begin{abstract}
In the paper, the authors find some new integral inequalities of Hermite-Hadamard type for functions whose derivatives of the $n$-th order are $(\alpha,m)$-convex and deduce some known results. As applications of the newly-established results, the authors also derive some inequalities involving special means of two positive real numbers.
\end{abstract}



\maketitle

\section{Introduction}

It is common knowledge in mathematical analysis that a function $f:I\subseteq\mathbb{R}\to \mathbb{R}$ is said to be convex on an interval $I\ne\emptyset$ if
\begin{equation}\label{convex-dfn-ineq}
f(\lambda x+(1-\lambda) y)\le \lambda f(x)+(1-\lambda )f(y)
\end{equation}
for all $x,y\in I$ and $\lambda \in [0,1]$; If the inequality~\eqref{convex-dfn-ineq} reverses, then $f$ is said to be concave on $I$.
\par
Let $f: I\subseteq\mathbb{R} \to \mathbb{R}$ be a convex function on an interval $I$ and $a, b\in I$ with $a<b$. Then
\begin{equation}\label{HH-ineq-eq}
f\biggl(\frac{a+b}2\biggl)\le\frac{1}{b-a}\int_a^bf(x)\td x\le \frac{f(a)+f(b)}2.
\end{equation}
This inequality is well known in the literature as Hermite-Hadamard integral inequality for convex functions. See~\cite{Dragomir-selected-Topic, Niculescu-Persson-Monograph-2004} and closely related references therein.
\par
The concept of usually used convexity has been generalized by a number of mathematicians. Some of them can be recited as follows.

\begin{definition}[\cite{Toader-1985-329}]
Let $f:[0,b]\to\mathbb{R}$ be a function and $m\in[0,1]$. If
\begin{equation}
 f(\lambda x +m(1-\lambda)y)\le\lambda f(x)+m(1-\lambda )f(y)
\end{equation}
holds for all $x,y\in[0,b]$ and $\lambda\in[0,1]$, then we say that $f(x)$ is $m$-convex on $[0,b]$.
\end{definition}

\begin{definition}[\cite{Mihesan-1993-Romania}]
Let $f:[0,b]\to\mathbb{R}$ be a function and $(\alpha,m)\in[0,1]\times[0,1]$. If
\begin{equation}
 f(\lambda x+m(1-\lambda)y)\le \lambda^\alpha f(x)+m(1-\lambda^\alpha)f(y)
\end{equation}
is valid for all $x,y\in[0,b]$ and $\lambda\in(0,1]$, then we say that $f(x)$ is $(\alpha, m)$-convex on $[0,b]$.
\end{definition}

It is not difficult to see that when $(\alpha,m)\in\{(\alpha,0),(1,0),(1,m),(1,1),(\alpha,1)\}$ the $(\alpha, m)$-convex function becomes the $\alpha$-star-shaped, star-shaped, $m$-convex, convex, and $\alpha$-convex functions respectively.
\par
The famous Hermite-Hadamard inequality~\eqref{HH-ineq-eq} has been refined or generalized by many mathematicians. Some of them can be reformulated as follows.

\begin{theorem}[{\cite[Theorem~3]{M.E.-2010-1065}}]
Let $f:I^\circ\subset[0,\infty)\to\mathbb{R}$ be a twice differentiable function such that $f''\in L([a,b])$ for $a,b\in I$ with $a<b$. If $|f''(x)|^{q}$ is $m$-convex on $[a,b]$ for some fixed $q>1$ and $m\in[0,1]$, then
\begin{equation}
\biggl|f\biggl(\frac{a+b}{2}\biggr)-\frac{1}{b-a} \int_{a}^{b}f(x)\td x\biggr|
\le \frac{(b-a)^{2}}{8}\biggl[\frac{\Gamma(1+p)}{\Gamma(3/2+p)}\biggr]^{1/p} \biggl[\frac{|f''(a)|^{q}+m|f''(b/m)|^{q}}{2}\biggr]^{1/q},
\end{equation}
where $\frac1p+\frac1q=1$ and $\Gamma$ is the classical Euler gamma function which may be defined for $\Re(z)>0$ by
\begin{equation}
\Gamma(z)=\int_0^{\infty}t^{z-1}e^{-t}\td t.
\end{equation}
\end{theorem}

\begin{theorem}[{\cite[Theorem~4]{Sarikaya-Aktan-1005.2897}}]
Let $I\subseteq\mathbb{R}$ be an open interval and $a,b\in I$ with $a<b$, and let $f:I\to\mathbb{R}$ be a twice differentiable mapping such that $f''(x)$ is integrable. If $0\le\lambda\le1$ and  $|f''(x)|$ is convex on $[a,b]$, then
\begin{multline}
\biggl|(\lambda-1)f\biggl(\frac{a+b}2\biggr)-\lambda\frac{f(a)+f(b)}2+\int_a^bf(x)\td x\biggr|\\
\le
\begin{cases}
\begin{aligned}
\frac{(b-a)^2}{24}\biggl\{\biggl[\lambda^4&+(1+\lambda)(1-\lambda)^3+\frac{5\lambda-3}4\biggr]|f''(a)|\\ &+\biggl[\lambda^4+(2-\lambda)\lambda^3+\frac{1-3\lambda}4\biggr]|f''(b)|\biggr\}, \quad 0\le\lambda\le\dfrac12;
\end{aligned}\\
\dfrac{(b-a)^2}{48}(3\lambda-1)\bigl(|f''(a)|+|f''(b)|\bigr),\quad \dfrac12\le\lambda\le1.
\end{cases}
\end{multline}
\end{theorem}

\begin{theorem}[{\cite[Theorem~3]{M.E.-2011-2614}}]\label{thm1.3}
Let $b^*>0$ and $f:[0,b^*]\to\mathbb{R}$ be a twice differentiable function such that $f''\in L([a,b])$ for $a,b\in[0,b^*]$ with $a<b$. If $|f''(x)|^{q}$ is $(\alpha,m)$-convex on $[a,b]$ for $(\alpha,m)\in[0,1]\times[0,1]$ and $q\ge1$, then
\begin{multline}\label{thm1.3-n=2}
\biggl|\frac{f(a)+f(mb)}2-\frac{1}{mb-a} \int_{a}^{mb}f(x)\td x\biggr|\\
\le \frac{(mb-a)^2}2\biggl(\frac1{6}\biggr)^{1-1/q}
\biggl\{\frac{|f''(a)|^{q}}{(\alpha+2)(\alpha+3)}+m|f''(b)|^{q} \biggl[\frac16-\frac1{(\alpha+2)(\alpha+3)}\biggr]\biggr\}^{1/q}.
\end{multline}
\end{theorem}

In recent years, some other kinds of Hermite-Hadamard type inequalities were generated in~\cite{Hadramard-Convex-Xi-Filomat.tex, H-H-Bai-Wang-Qi-2012.tex, chun-ling-Hermite.tex, difference-hermite-hadamard.tex, GMJ-2013-062.tex, 130-2014-Shuang-Wang-Qi-JOCAAA-2-10-2014.tex, Xi-Bai-Qi-Hadamard-2011-AEQMath.tex, H-H-h-C-Xi-Qi-AIA.tex, Hadramard-Convex-Xi-September-2011.tex}, for example. For more systematic information, please refer to monographs~\cite{Dragomir-selected-Topic, Niculescu-Persson-Monograph-2004} and related references therein.
\par
In this paper, we will establish some new inequalities of Hermite-Hadamard type for functions whose derivatives of $n$-th order are $(\alpha,m)$-convex and deduce some known results in the form of corollaries.

\section{A lemma}

For establishing new integral inequalities of Hermite-Hadamard type for functions whose derivatives of $n$-th order are $(\alpha,m)$-convex, we need the following lemma.

\begin{lemma}\label{lem2.1}
Let $0<m\le1$ and $b>a>0$ satisfying $a<mb$.
If $f^{(n)}(x)$ for $n\in\{0\}\cup\mathbb{N}$ exists and is integrable on the closed interval $[0,b]$, then
\begin{multline}\label{eq2.1}
\frac{f(a)+f(mb)}{2} -\frac{1}{mb-a}\int_{a}^{mb}f(x)\td x -\frac12\sum_{k=2}^{n-1}\frac{(k-1)(mb-a)^k}{(k+1)!}f^{(k)}(a) \\*
=\frac12\frac{(mb-a)^n}{n!}\int_{0}^1t^{n-1}(n-2t)f^{(n)}(ta+m(1-t)b)\td t,
\end{multline}
where the sum above takes $0$ when $n=1$ and $n=2$.
\end{lemma}

\begin{proof}
When $n=1$, it is easy to deduce the identity~\eqref{eq2.1} by performing an integration by parts in the integrals from the right side and changing the variable.
\par
When $n=2$, we have
\begin{equation}\label{eq2.1-n=2}
\frac{f(a)+f(mb)}{2}-\frac{1}{mb-a}\int_{a}^{mb}f(x)\td x
=\frac{(mb-a)^2}2\int_{0}^1t(1-t)f''(ta+m(1-t)b)\td t.
\end{equation}
This result is same as~\cite[Lemma~2]{M.E.-2011-2614}.
\par
When $n=3$, the identity~\eqref{eq2.1} is equivalent to
\begin{multline}\label{eq2.1-n=3}
\frac{f(a)+f(mb)}{2}-\frac{1}{mb-a}\int_{a}^{mb}f(x)\td x-\frac{(mb-a)^2}{12}f''(a)\\
=\frac{(mb-a)^3}{12}\int_{0}^1t^2(3-2t)f^{(3)}(ta+m(1-t)b)\td t,
\end{multline}
which may be derived from integrating the integral in the second line of~\eqref{eq2.1-n=3} and utilizing the identity~\eqref{eq2.1-n=2}.
\par
When $n\ge4$, computing the second line in~\eqref{eq2.1} by integration by parts yields
\begin{multline*}
\frac{(mb-a)^n}{n!}\int_{0}^1 t^{n-1}(n-2t)f^{(n)}(ta+m(1-t)b)\td t\\
=-\frac{(n-2)(mb-a)^{n-1}}{n!}f^{(n-1)}(a)
+\frac{(mb-a)^{n-1}}{(n-1)!} \int_{0}^1t^{n-2}(n-1-2t)f^{(n-1)}(ta+m(1-t)b)\td t,
\end{multline*}
which is a recurrent formula
\begin{equation*}
S_{a,mb}(n)=-T_{a,mb}(n-1)+S_{a,mb}(n-1)
\end{equation*}
on $n$, where
\begin{equation*}
S_{a,mb}(n)=\frac12\frac{(mb-a)^n}{n!}\int_{0}^1 t^{n-1}(n-2t)f^{(n)}(ta+m(1-t)b)\td t
\end{equation*}
and
\begin{equation*}
T_{a,mb}(n-1)=\frac12\frac{(n-2)(mb-a)^{n-1}}{n!}f^{(n-1)}(a)
\end{equation*}
for $n\ge4$. By mathematical induction, the proof of Lemma~\ref{lem2.1} is complete.
\end{proof}

\begin{remark}
Similar integral identities to~\eqref{eq2.1}, produced by replacing $f^{(k)}(a)$ in~\eqref{eq2.1} by $f^{(k)}(b)$ or by $f^{(k)}\bigl(\frac{a+b}2\bigr)$, and corresponding integral inequalities of Hermite-Hadamard type have been established in~\cite{H-H-(a-m)-convex-Filomat.tex, Wang-Qi-MIA3459-MINFAA2012.tex, Wang-Ineq-H-H-type-Analysis.tex}.
\end{remark}

\begin{remark}
When $m=1$, our Lemma~\ref{lem2.1} becomes~\cite[Lemma~2.1]{Hwang-Kyugpook-03}.
\end{remark}

\section{Inequalities of Hermite-Hadamard type}

Now we are in a position to establish some integral inequalities of Hermite-Hadamard type for functions whose derivatives of $n$-th order are $(\alpha,m)$-convex.

\begin{theorem}\label{th3.1}
Let $(\alpha,m)\in[0,1]\times(0,1]$ and $b>a>0$ with $a<mb$.
If $f(x)$ is $n$-time differentiable on $[0,b]$ such that $\bigl|f^{(n)}(x)\bigr|\in L([0,mb])$ and $\bigl|f^{(n)}(x)\bigr|^p$ is $(\alpha,m)$-convex on $[0,mb]$ for $n\ge 2$ and $p\ge 1$, then
\begin{multline}\label{eq3.1.1}
\biggl|\frac{f(a)+f(mb)}{2}-\frac{1}{mb-a}\int_{a}^{mb}f(x)\td x -\frac12\sum_{k=2}^{n-1}\frac{(k-1)(mb-a)^k}{(k+1)!}f^{(k)}(a)\biggr|\\*
\le\frac12\frac{(mb-a)^n}{n!}\biggl(\frac{n-1}{n+1}\biggr)^{1-1/p} \biggl\{\frac{n(n-1)+\alpha(n-2)}{(n+\alpha)(n+\alpha+1)}\bigl|f^{(n)}(a)\bigr|^p\\ +m\biggl[\frac{n-1}{n+1}-\frac{n(n-1)+\alpha(n-2)} {(n+\alpha)(n+\alpha+1)}\biggr]\bigl|f^{(n)}(b)\bigr|^p\biggr\}^{1/p},
\end{multline}
where the sum above takes $0$ when $n=2$.
\end{theorem}
\begin{proof}
It follows from Lemma~\ref{eq2.1} that
\begin{multline}\label{eq3.1.2}
\biggl|\frac{f(a)+f(mb)}{2}-\frac{1}{mb-a}\int_{a}^{mb}f(x)\td x -\frac12\sum_{k=2}^{n-1}\frac{(k-1)(mb-a)^k}{(k+1)!}f^{(k)}(a)\biggr|\\
\le \frac12\frac{(mb-a)^n}{n!}\int_{0}^1t^{n-1}(n-2t)\bigl|f^{(n)}(ta+m(1-t)b)\bigr|\td t.
\end{multline}
\par
When $p=1$, since $\bigl|f^{(n)}(x)\bigr|$ is $(\alpha,m)$-convex, we have
\begin{equation*}
\bigl|f^{(n)}(ta+m(1-t)b)\bigr|\le t^\alpha\bigl|f^{(n)}(a)\bigr|+m(1-t^\alpha)\bigl|f^{(n)}(b)\bigr|.
\end{equation*}
Multiplying by the factor $t^{n-1}(n-2t)$ on both sides of the above inequality and integrating with respect to $t\in[0,1]$ lead to
\begin{align*}
&\quad\int_{0}^1 t^{n-1}(n-2t)\bigl|f^{(n)}(ta+m(1-t)b)\bigr|\td t\\
&\le\int_{0}^1t^{n-1}(n-2t)\bigl[t^{\alpha}\bigl|f^{(n)}(a)\bigr| +m(1-t^{\alpha})\bigl|f^{(n)}(b)\bigr|\bigr]\td t\\
&=\bigl|f^{(n)}(a)\bigr|\int_{0}^1 t^{n+\alpha-1}(n-2t)\td t +m\bigl|f^{(n)}(b)\bigr|\int_{0}^1t^{n-1}(n-2t)(1-t^{\alpha})\td t\\
&=\biggl(\frac{n}{n+\alpha}-\frac2{n+\alpha+1}\biggr)\bigl|f^{(n)}(a)\bigr|
+m\bigl|f^{(n)}(b)\bigr|\biggl(\frac{n-1}{n+1}-\frac{n}{n+\alpha}+\frac2{n+\alpha+1}\biggr)\\
&=\frac{n(n-1)+\alpha(n-2)}{(n+\alpha)(n+\alpha+1)}\bigl|f^{(n)}(a)\bigr| +m\biggl[\frac{n-1}{n+1}-\frac{n(n-1)+\alpha(n-2)}{(n+\alpha)(n+\alpha+1)}\biggr]\bigl|f^{(n)}(b)\bigr|.
\end{align*}
The proof for the case $p=1$ is complete.
\par
When $p>1$, by the well-known H\"older integral inequality, we obtain
\begin{multline}\label{eq3.1.3}
\int_{0}^1t^{n-1}(n-2t)\bigl|f^{(n)}(ta+m(1-t)b)\bigr|\td t\\*
\le\biggl[\int_{0}^1t^{n-1}(n-2t)\td t\biggr]^{1-1/p}
\biggl[\int_{0}^1t^{n-1}(n-2t)\bigl|f^{(n)}(ta+m(1-t)b)\bigr|^p\td t\biggr]^{1/p}.
\end{multline}
Using the $(\alpha,m)$-convexity of $\bigl|f^{(n)}(x)\bigr|^p$ produces
\begin{multline}\label{eq3.1.4}
\int_{0}^1t^{n-1}(n-2t)\bigl|f^{(n)}(ta+m(1-t)b)\bigr|^p \td t\\
\le\int_{0}^1t^{n-1}(n-2t)\bigl[t^{\alpha}\bigl|f^{(n)}(a)\bigr|^p
+m(1-t^{\alpha})\bigl|f^{(n)}(b)\bigr|^p\bigr]\td t\\
=\frac{n(n-1)+\alpha(n-2)}{(n+\alpha)(n+\alpha+1)}\bigl|f^{(n)}(a)\bigr|^p
 +m\biggl[\frac{n-1}{n+1}-\frac{n(n-1)+\alpha(n-2)}{(n+\alpha)(n+\alpha+1)}\biggr]\bigl|f^{(n)}(b)\bigr|^p.
\end{multline}
Substituting~\eqref{eq3.1.3} and~\eqref{eq3.1.4} into~\eqref{eq3.1.2} yields the inequality~\eqref{eq3.1.1}.
This completes the proof of Theorem~\ref{th3.1}.
\end{proof}

\begin{corollary}\label{cor3.1}
Under conditions of Theorem~\ref{th3.1},
\begin{enumerate}
\item
when $m=1$, we have
\begin{multline*}
\biggl|\frac{f(a)+f(b)}{2}-\frac{1}{b-a}\int_{a}^{b}f(x)\td x -\frac12\sum_{k=2}^{n-1}\frac{(k-1)(b-a)^k}{(k+1)!}f^{(k)}(a)\biggr|
\le\frac12\frac{(b-a)^n}{n!}\biggl(\frac{n-1}{n+1}\biggr)^{1-1/p}\\
\times\biggl\{\frac{n(n-1)+\alpha(n-2)}{(n+\alpha)(n+\alpha+1)}\bigl|f^{(n)}(a)\bigr|^p
+\biggl[\frac{n-1}{n+1}-\frac{n(n-1) +\alpha(n-2)}{(n+\alpha)(n+\alpha+1)}\biggr]\bigl|f^{(n)}(b)\bigr|^p\biggr\}^{1/p};
\end{multline*}
\item
when $n=2$, we have
\begin{multline*}
\biggl|\frac{f(a)+f(mb)}{2}-\frac{1}{mb-a}\int_{a}^{mb}f(x)\td x \biggr|\\
\le\frac{(mb-a)^2}4\biggl(\frac13\biggr)^{1-1/p}
\biggl\{\frac2{(\alpha+2)(\alpha+3)}\bigl|f''(a)\bigr|^p +m\biggl[\frac13-\frac2{(\alpha+2)(\alpha+3)}
\biggr]\bigl|f''(b)\bigr|^p\biggr\}^{1/p};
\end{multline*}
\item
when $m=\alpha=p=1$ and $n=2$, we have
\begin{equation*}
\biggl|\frac{f(a)+f(b)}{2}-\frac{1}{b-a}\int_{a}^{b}f(x)\td x \biggr|
\le\frac{(b-a)^2}{24}\bigl[\bigl|f''(a)\bigr|+\bigl|f''(b)\bigr|\bigr];
\end{equation*}
\item
when $m=\alpha=1$ and $p=n=2$, we have
\begin{equation*}
\biggl|\frac{f(a)+f(b)}{2}-\frac{1}{b-a}\int_{a}^{b}f(x)\td x \biggr|
\le\frac{(b-a)^2}{12}\biggl[\frac{|f''(a)|^2+|f''(b)|^2}2\biggr]^{1/2}.
\end{equation*}
\end{enumerate}
\end{corollary}

\begin{remark}
Under conditions of Theorem~\ref{th3.1},
\begin{enumerate}
\item
when $n=2$, the inequality~\eqref{eq3.1.1} becomes the one~\eqref{thm1.3-n=2} in~\cite[Theorem~3]{M.E.-2011-2614};
\item
when $\alpha=m=1$, Theorem~\ref{th3.1} becomes~\cite[Theorem~3.1]{Hwang-Kyugpook-03}.
\end{enumerate}
\end{remark}

\begin{theorem}\label{th3.2}
Let $(\alpha,m)\in[0,1]\times(0,1]$ and $b>a>0$ with $a<mb$.
If $f(x)$ is $n$-time differentiable on $[0,b]$ such that $\bigl|f^{(n)}(x)\bigr|\in L([0,mb])$ and $\bigl|f^{(n)}(x)\bigr|^p$ is $(\alpha,m)$-convex on $[0,mb]$ for $n\ge 2$ and $p>1$, then
\begin{multline}\label{th3.2-ineq}
\biggl|\frac{f(a)+f(mb)}{2}-\frac{1}{mb-a}\int_{a}^{mb}f(x)\td x -\frac12\sum_{k=2}^{n-1}\frac{(k-1)(mb-a)^k}{(k+1)!}f^{(k)}(a)\biggr|\\
\le\frac12\frac{(mb-a)^n}{n!}\biggl[\frac{n^{q+1}-(n-2)^{q+1}}{2(q+1)}\biggr]^{1/q} \biggl\{\frac1{p(n-1)+\alpha+1}\bigl|f^{(n)}(a)\bigr|^p\\*
+\frac{m\alpha}{[p(n-1)+1][p(n-1)+\alpha+1]}\bigl|f^{(n)}(b)\bigr|^p\biggr\}^{1/p},
\end{multline}
where the sum above takes $0$ when $n=2$ and $\frac1p+\frac1q=1$.
\end{theorem}

\begin{proof}
It follows from Lemma~\ref{lem2.1} that
\begin{multline}\label{eq3.2.2}
\biggl|\frac{f(a)+f(mb)}{2}-\frac{1}{mb-a}\int_{a}^{mb}f(x)\td x -\frac12\sum_{k=2}^{n-1}\frac{(k-1)(mb-a)^k}{(k+1)!}f^{(k)}(a)\biggr|\\
\le \frac12\frac{(mb-a)^n}{n!}\int_{0}^1t^{n-1}(n-2t)\bigl|f^{(n)}(ta+m(1-t)b)\bigr|\td t.
\end{multline}
By the well-known H\"older integral inequality, we obtain
\begin{multline}\label{eq3.2.3}
\int_{0}^1t^{n-1}(n-2t)\bigl|f^{(n)}(ta+m(1-t)b)\bigr|\td t\\
\le\biggl[\int_{0}^1(n-2t)^q\td t\biggr]^{1/q} \biggl[\int_{0}^1t^{p(n-1)}\bigl|f^{(n)}(ta+m(1-t)b)\bigr|^p\td t\biggr]^{1/p}\\
=\biggl[\frac{n^{q+1}-(n-2)^{q+1}}{2(q+1)}\biggr]^{1/q} \biggl[\int_{0}^1t^{p(n-1)}\bigl|f^{(n)}(ta+m(1-t)b)\bigr|^p\td t\biggr]^{1/p}.
\end{multline}
Making use of the $(\alpha,m)$-convexity of $\bigl|f^{(n)}(x)\bigr|^p$ reveals
\begin{multline}\label{eq3.2.4}
\int_{0}^1t^{p(n-1)}\bigl|f^{(n)}(ta+m(1-t)b)\bigr|^p \td t \\*
\begin{aligned}
&\le\int_{0}^1t^{p(n-1)}\bigl[t^{\alpha}\bigl|f^{(n)}(a)\bigr|^p +m(1-t^{\alpha})\bigl|f^{(n)}(b)\bigr|^p\bigr]\td t\\
&=\bigl|f^{(n)}(a)\bigr|^p\int_{0}^1t^{p(n-1)+\alpha}\td t +m\bigl|f^{(n)}(b)\bigr|^p\int_{0}^1t^{p(n-1)}(1-t^{\alpha})\td t
\end{aligned}\\
=\frac{\bigl|f^{(n)}(a)\bigr|^p}{p(n-1)+\alpha+1} +\frac{m\alpha}{[p(n-1)+1][p(n-1)+\alpha+1]}\bigl|f^{(n)}(b)\bigr|^p.
\end{multline}
Combining~\eqref{eq3.2.3} and~\eqref{eq3.2.4} with~\eqref{eq3.2.2} results in the inequality~\eqref{th3.2-ineq}.
This completes the proof of Theorem~\ref{th3.2}.
\end{proof}

\begin{corollary}
Under conditions of Theorem~\ref{th3.2},
\begin{enumerate}
\item
when $m=1$, we have
\begin{multline*}
\biggl|\frac{f(a)+f(b)}2-\frac{1}{b-a}\int_{a}^{b}f(x)\td x -\frac12\sum_{k=2}^{n-1}\frac{(k-1)(b-a)^k}{(k+1)!}f^{(k)}(a)\biggr|\\*
\le\frac12\frac{(b-a)^n}{n!}\biggl[\frac{n^{q+1}-(n-2)^{q+1}}{2(q+1)}\biggr]^{1/q} \biggl\{\frac1{p(n-1)+\alpha+1}\bigl|f^{(n)}(a)\bigr|^p\\*
+\frac{\alpha}{[p(n-1)+1][p(n-1)+\alpha+1]}\bigl|f^{(n)}(b)\bigr|^p\biggr\}^{1/p};
\end{multline*}
\item
when $n=2$, we have
\begin{multline*}
\biggl|\frac{f(a)+f(mb)}{2}-\frac{1}{mb-a}\int_{a}^{mb}f(x)\td x \biggr|\\
\le\frac{(mb-a)^2}2\biggl(\frac1{q+1}\biggr)^{1/q}
\biggl[\frac1{p+\alpha+1}\bigl|f''(a)\bigr|^p+\frac{m\alpha}
{(p+1)(p+\alpha+1)}\bigl|f''(b)\bigr|^p\biggr]^{1/p};
\end{multline*}
\item
when $m=\alpha=1$ and $n=2$, we have
\begin{equation} \label{e}
\biggl| \frac{f(a)  +f(b)  }{2}-\frac{1}{b-a}\int_{a}^{b}f(x)  \td x\biggr|
\le\frac{(b-a)^{2}}{2(p+1)^{1/p}(q+2)^{1/q}} \biggl[\frac{(q+1)|f''(a)|^{q}+|f''(b)|^{q}}{q+1}\biggr]^{1/q},
\end{equation}
where $\frac{1}{p}+\frac{1}{q}=1$.
\end{enumerate}
\end{corollary}

\begin{theorem}\label{th3.3}
Let $(\alpha,m)\in[0,1]\times(0,1]$ and $b>a>0$ with $a<mb$. If $f(x)$ is $n$-time differentiable on $[0,b]$ such that $\bigl|f^{(n)}(x)\bigr|\in L([0,mb])$ and $\bigl|f^{(n)}(x)\bigr|^p$ is $(\alpha,m)$-convex on $[0,mb]$ for $n\ge 2$ and $p\ge 1$, then
\begin{multline}
\biggl|\frac{f(a)+f(mb)}{2}-\frac{1}{mb-a}\int_{a}^{mb}f(x)\td x -\frac12\sum_{k=2}^{n-1}\frac{(k-1)(mb-a)^k}{(k+1)!}f^{(k)}(a)\biggr|\\
\begin{aligned}
&\le\frac{(n-1)^{1-1/p}}2\frac{(mb-a)^n}{n!} \biggl\{\frac{(n-2)(pn-p+\alpha)+2(n-1)} {(pn-p+\alpha+1)(pn-p+\alpha+2)}\bigl|f^{(n)}(a)\bigr|^p\\
&\quad+m\bigg[\frac{(n-1)(pn-2p+2)}{(pn-p+1)(pn-p+2)}-
\frac{(n-2)(pn-p+\alpha)+2(n-1)} {(pn-p+\alpha+1)(pn-p+\alpha+2)}\biggr]\bigl|f^{(n)}(b)\bigr|^p\biggr\}^{1/p},
\end{aligned}
\end{multline}
where the sum above takes $0$ when $n=2$.
\end{theorem}

\begin{proof}
Utilizing Lemma~\ref{eq2.1}, H\"older integral inequality, and the $(\alpha,m)$-convexity of $\bigl|f^{(n)}(x)\bigr|^p$ yields
\begin{align*}
&\quad\biggl|\frac{f(a)+f(mb)}{2}-\frac{1}{mb-a}\int_{a}^{mb}f(x)\td x -\frac12\sum_{k=2}^{n-1}\frac{(k-1)(mb-a)^k}{(k+1)!}f^{(k)}(a)\biggr|\\
&\le\frac12\frac{(mb-a)^n}{n!}\int_{0}^1t^{n-1}(n-2t)\bigl|f^{(n)}(ta+m(1-t)b)\bigr|\td t\\
&\le\frac12\frac{(mb-a)^n}{n!}\biggl[\int_{0}^1(n-2t)\td t\biggr]^{1-1/p}\\
&\quad\times\biggl\{\int_{0}^1t^{p(n-1)}(n-2t)\bigl[t^{\alpha}|f^{(n)}(a)|^p +m(1-t^{\alpha})|f^{(n)}(b)|^p\bigr]\td t\biggr\}^{1/p}\\
&=\frac{(n-1)^{1-1/p}}2\frac{(mb-a)^n}{n!} \biggl\{\frac{(n-2)(pn-p+\alpha)+2(n-1)}{(pn-p+\alpha+1)(pn-p+\alpha+2)}\bigl|f^{(n)}(a)\bigr|^p+m\\
&\times\bigg[\frac{(n-1)(pn-2p+2)}{(pn-p+1)(pn-p+2)}-\frac{(n-2)(pn-p+\alpha)+2(n-1)}{(pn-p+\alpha+1)
(pn-p+\alpha+2)}\biggr]\bigl|f^{(n)}(b)\bigr|^p\biggr\}^{1/p}.
\end{align*}
This completes the proof of Theorem~\ref{th3.3}.
\end{proof}

\begin{corollary}
Under conditions of Theorem~\ref{th3.3},
\begin{enumerate}
\item
when $m=1$, we have
\begin{multline*}
\biggl|\frac{f(a)+f(b)}{2}-\frac{1}{b-a}\int_{a}^{b}f(x)\td x -\frac12\sum_{k=2}^{n-1}\frac{(k-1)(b-a)^k}{(k+1)!}f^{(k)}(a)\biggr|\\
\le\frac{(n-1)^{1-1/p}}2\frac{(b-a)^n}{n!} \biggl\{\frac{(n-2)(pn-p+\alpha)+2(n-1)}{(pn-p+\alpha+1)(pn-p+\alpha+2)}\bigl|f^{(n)}(a)\bigr|^p\\
+\biggl[\frac{(n-1)(pn-2p+2)}{(pn-p+1)(pn-p+2)}
-\frac{(n-2)(pn-p+\alpha)+2(n-1)}{(pn-p+\alpha+1)(pn-p+\alpha+2)}\biggr]\bigl|f^{(n)}(b)\bigr|^p\biggr\}^{1/p};
\end{multline*}
\item
when $n=2$, we have
\begin{multline*}
\biggl|\frac{f(a)+f(mb)}{2}-\frac{1}{mb-a}\int_{a}^{mb}f(x)\td x\biggr|
\le\frac{(mb-a)^2}4\biggl\{\frac2{(p+\alpha+1)(p+\alpha+2)}\bigl|f''(a)\bigr|^p\\
+m\biggl[\frac2{(p+1)(p+2)}-\frac2{(p+\alpha+1)(p+\alpha+2)}\biggr]\bigl|f''(b)\bigr|^p\biggr\}^{1/p};
\end{multline*}
\item
when $m=\alpha=1$ and $n=2$, we have
\begin{equation} \label{k}
\biggl|\frac{f(a)  +f(b)  }{2}-\frac{1}{b-a}\int_{a}^{b}f(x)  \td x\biggr|\le\frac{(b-a)  ^{2}}{2^{2-1/p}}\biggl[  \frac{(p+1)  |f''(a)|^{p}+2|f''(b)|^{p}}{(p+1)(p+2)(p+3)}\biggr]^{1/p}.
\end{equation}
\end{enumerate}
\end{corollary}

\section{Applications to special means}

It is well known that, for positive real numbers $\alpha$ and $\beta$ with $\alpha\ne\beta$, the quantities
\begin{gather*}
A(\alpha,\beta)=\frac{\alpha+\beta}{2},\quad G(\alpha,\beta)=\sqrt{\alpha\beta}\,, \quad H( \alpha,\beta) =\frac{2}{1/\alpha+1/\beta},\\
I(\alpha,\beta)=\frac{1}{e}\biggl(\frac{\beta^{\beta}}{\alpha^{\alpha}}\biggr)^{1/(\beta-\alpha)}, \quad
L(\alpha,\beta)=\frac{\alpha-\beta}{\ln\alpha-\ln\beta},\quad
L_{r}(\alpha,\beta)=\biggl[  \frac{\beta^{r+1}-\alpha^{r+1}}{(r+1)(\beta-\alpha)}\biggr]^{1/r}
\end{gather*}
for $r\ne0,-1$ are respectively called the arithmetic, geometric, harmonic, exponential, logarithmic, and generalized logarithmic means.
\par
Basing on inequalities of Hermite-Hadamard type in the above section, we shall derive some inequalities of the above defined means as follows.

\begin{theorem}\label{Prop1}
Let $r\in(-\infty,0)\cup[1,\infty)\setminus\{-1\}$ and $b>a>0$. Then, for $p,q>1$,
\begin{equation}\label{m}
|A(a^{r},b^{r})-[L_{r}(a,b)]^{r}|\le\frac{(b-a)^{2}r(r-1)}{2(p+1)^{1/p}(q+2)^{1/q}} \biggl[a^{(r-2)q}+\frac{b^{(r-2)q}}{q+1}\biggr]^{1/q},
\end{equation}
where $\frac1p+\frac1q=1$.
\end{theorem}

\begin{proof}
This follows from applying the inequality~\eqref{e} to the function $f(x)=x^{r}$.
\end{proof}

\begin{theorem}\label{Prop2}
Let $r\in(-\infty,0)\cup[1,\infty)\setminus\{-1\}$ and $b>a>0$. Then, for $p\ge1$,
\begin{equation}\label{l}
| A(a^{r},b^{r})-[L_{r}(a,b)]^{r}| \le\frac{(b-a)^{2}r(r-1)}{2^{2-1/p}} \biggl[\frac{(p+1)a^{(r-2)p} +2b^{(r-2)p}}{(p+1)(p+2)(p+3)}\biggr]^{1/p}.
\end{equation}
\end{theorem}

\begin{proof}
This follows from applying the inequality~\eqref{k} to the function $f(x)=x^{r}$.
\end{proof}

\begin{theorem}\label{Prop3}
Let $r\in(-\infty,0)\cup[1,\infty)\setminus\{-1\}$ and $b>a>0$. Then
\begin{equation}\label{n}
| A(a^{r},b^{r})-[L_{r}(a,b)]^{r}| \le\frac{(b-a)^{2}r(r-1)}{24}A\bigl(a^{r-2},b^{r-2}\bigr).
\end{equation}
\end{theorem}

\begin{proof}
This follows from applying the inequality~\eqref{k} for $p=1$ to the function $f(x)=x^{r}$.
\end{proof}

\begin{theorem}\label{Prop4}
Let $b>a>0$. Then for $p,q>1$ we have
\begin{equation}\label{o}
\biggl| \frac1{H(a,b)}-\frac1{L(a,b)}\biggr|
\le\frac{(b-a)^{2}}{(p+1)^{1/p}(q+2)^{1/q}}\biggl[\frac1{a^{3q}}+\frac1{(q+1)b^{3q}}\biggr]^{1/q},
\end{equation}
where $\frac1p+\frac1q=1$.
\end{theorem}

\begin{proof}
This follows from applying the inequality~\eqref{e} to the function $f(x)=\frac{1}{x}$.
\end{proof}

\begin{theorem}\label{Prop5}
Let $b>a>0$. Then for $p\ge1$ we have
\begin{equation}
\biggl| \frac1{H(a,b)}-\frac1{L(a,b)}\biggr|
\le\frac{(b-a)^{2}r(r-1)}{2^{1-1/p}[(p+2)(p+3)]^{1/p}} \biggl[\frac1{a^{3p}}+\frac2{(p+1)b^{3p}}\biggr]^{1/p}.
\end{equation}
\end{theorem}

\begin{proof}
This follows from the inequality~\eqref{k} to the function $f(x)=x^{r}$.
\end{proof}

\begin{theorem}\label{Prop6}
Let $b>a>0$. Then we have
\begin{equation}\label{r}
\ln \frac{I(a,b)}{G(a,b)}
\le\frac{(b-a)^{2}}{24}A\biggl(\frac1{a^{2}},\frac1{b^{2}}\biggr).
\end{equation}
\end{theorem}

\begin{proof}
This follows from applying the inequality~\eqref{k} for $p=1$ to the function $f(x)=-\ln x$.
\end{proof}

\begin{remark}
This paper is a combined version of the preprints~\cite{n-times-diferentiable-functions_m-convex.tex, try-too-to-arxiv.tex}.
\end{remark}

\subsection*{Acknowledgements}
The authors would like to thank Professors Bo-Yan Xi and Shu-Hong Wang at Inner Mongolia University for Nationalities in China for their helpful corrections to and valuable comments on the original version of this paper.

\end{document}